\documentclass[12pt,a4paper]{article}

\usepackage{graphicx}
\usepackage{subfig}
\usepackage{amsmath,amsthm,amssymb}
\usepackage{esint}
\usepackage{mathrsfs}
\usepackage{appendix}
\usepackage{enumitem}

\usepackage{hyperref}

\usepackage[left=2.2cm,right=2.2cm,top=3cm,bottom=3.5cm]{geometry}

\usepackage{color}
\usepackage{array}
\usepackage{url}

\usepackage{float}

\newtheorem{theorem}{Theorem}[section]
\newtheorem*{theorem*}{Theorem}
\newtheorem{lemma}[theorem]{Lemma}

\theoremstyle{definition}

\theoremstyle{remark}
\newtheorem{remark}[theorem]{Remark}
\usepackage{authblk}

\usepackage{mathtools}
\usepackage{amssymb}
\usepackage{appendix}
\usepackage{multirow}

\numberwithin{equation}{section}

\usepackage[numbers,sort&compress]{natbib}

\newcommand{\dx}{\ {\rm d} {x}}
\newcommand{\dt}{\ {\rm d} t }

\DeclareMathOperator{\Tr}{Tr}

\allowdisplaybreaks

\begin{document}

\title{A mathematical study of an elastic-viscous-plastic sea-ice model with the Kelvin-Voigt rheology}
\author{Daniel W. Boutros\footnote{Department of Applied Mathematics and Theoretical Physics, University of Cambridge, Cambridge CB3 0WA UK. Email: \textsf{dwb42@cam.ac.uk}}\space,\space Xin Liu\footnote{Department of Mathematics, Texas A\&M University, College Station, TX 77843-3368, USA. Email: \textsf{xliu23@tamu.edu}}\space,\space Marita Thomas\footnote{Department of Mathematics and Computer Science, Freie Universität Berlin, Arnimallee 9, 14195 Berlin, Germany. Email: \textsf{marita.thomas@fu-berlin.de}}\space, and Edriss S. Titi\footnote{Department of Applied Mathematics and Theoretical Physics, University of Cambridge, Cambridge CB3 0WA UK; Department of Mathematics, Texas A\&M University, College Station, TX 77843-3368, USA; also Department of Computer Science and Applied Mathematics, Weizmann Institute of Science, Rehovot 76100, Israel. Emails: \textsf{Edriss.Titi@maths.cam.ac.uk} \; \textsf{titi@math.tamu.edu}}}

\date{April 29, 2026}

\maketitle

\begin{abstract}
Motivated by the elastic-viscous-plastic (EVP) sea-ice model [E. C. Hunke and J. K. Dukowicz, J. Phys. Oceanogr., 27, 9 (1997), 1849--1867], which is used in large-scale numerical climate simulations, we proposed in [D. W. Boutros, X. Liu, M. Thomas and E. S. Titi, arXiv:2505.03080 (2025)] the use of the inviscid Voigt regularisation for the constitutive (stress-tensor) relation and proved the global well-posedness of the resulting model. The EVP model treats sea ice as a non-Newtonian fluid. In turn, elastic-viscous-plastic solids often involve a Kelvin-Voigt viscosity in terms of the strain rate. Therefore, in the present work we formulate an elastic-viscous-plastic sea-ice model with a Kelvin-Voigt regularisation in terms of the strain rate. In other words, we introduce the Voigt regularisation in the momentum balance rather than in the constitutive relation (for the stress tensor). We then prove the local well-posedness for the Kelvin-Voigt EVP model with the advection term, in the momentum balance, and the global well-posedness in the absence of the advection term (following a very standard approximation in the latter case). A crucial component of the proof of these results, is a new $L^\infty$-estimate for the stress tensor which relies on the damping structure. Note that, both with and without the advection term, we are able to handle the case of viscosity coefficients without a cutoff from above, which remains a major open problem for the closely related Hibler sea-ice model. We are also able to prove the existence of solutions for much less regular initial data compared to our previous paper on the Voigt-EVP model.
\end{abstract}


\vspace{0.1cm} \noindent \textbf{Mathematics Subject Classification:} 35Q86 (primary), 35A01, 35A02, 35A09, 35B65, 74D10, 74H20, 74H25, 74H30, 86A08, 86A40 (secondary)

\section{Introduction}
\subsection{Formulation of the EVP model}
The evolution of the sea-ice cover plays a fundamental role in the climate system, and adequate modelling of the dynamics of sea ice is therefore of great importance. In this paper, we consider the elastic-viscous-plastic (EVP) sea-ice model, which was originally introduced in \cite{hunke}. The EVP model, on the time interval $[0,T]$ and the two-dimensional flat torus $\mathbb{T}^2 = (\mathbb{R} / \mathbb{Z})^2$, is given by
\begin{subequations} \label{EVPsystem}
\begin{align}
&\partial_t u  = \nabla \cdot \sigma + \mathcal{T}_a + \mathcal{T}_w + \Omega u^\perp - g \nabla H_0, \label{EVP1} \\
&\dfrac{1}{\mathcal E} \partial_t \sigma + \dfrac{4 \mathcal{D}}{P} (\sigma - \frac{1}{2} \Tr \sigma \mathbb I_2) + \dfrac{\mathcal{D}}{2P} \Tr \sigma \mathbb I_2 + \dfrac{\mathcal{D}}{2} \mathbb I_2  = D(u), \label{EVP2}
\\
&u \lvert_{t = 0} = u_0, \quad \sigma \lvert_{t = 0} = \sigma_0, \label{EVP3}
\end{align}
\end{subequations}
where $P$ is the internal ice strength parameter, which is a given positive constant in our setting, while $u : \mathbb{T}^2 \times [0,T] \rightarrow \mathbb{R}^2$ is the velocity field, $\sigma : \mathbb{T}^2 \times [0,T] \rightarrow \mathbb{R}^{2 \times 2}_{\mathrm{sym}}$ is the stress tensor, which are the unknowns. The tensor field $D(u)$ is the deformation-rate tensor, i.e., the symmetric part of the velocity gradient
\begin{equation}
D(u) \coloneqq \frac{1}{2} \lbrack \nabla u + (\nabla u)^\top \rbrack.
 \end{equation}
In system \eqref{EVPsystem} above we have used the strain rate, which is defined (and simplified) as follows
\begin{equation}
\mathcal{D} = \lvert D(u) \rvert,
\end{equation}
where in this paper we will use the notation $\lvert \cdot \rvert$ for the Euclidean norm for tensors. In other words, for a tensor $(A_{ij})_{i,j=1}^2$ we define the norm $\lvert \cdot \rvert$ as follows
\begin{equation*}
\lvert A \rvert \coloneqq \sqrt{\sum_{i,j=1}^2 \lvert A_{ij} \rvert^2}.
\end{equation*}
In addition, $\mathcal{T}_a$ and $\mathcal{T}_w$ are the atmospheric and oceanic drag forces, which are given by
\begin{align}
\mathcal{T}_{a} & \coloneqq c_a \rho_a \lvert U_a \rvert \bigg(U_a \cos \phi + U_a^\perp \sin \phi \bigg), \label{atmosphericdrag} \\
\mathcal{T}_{w} & \coloneqq c_w \rho_w \lvert U_w - u \rvert \bigg[ (U_w - u) \cos \theta + (U_w - u)^\perp \sin \theta \bigg]. \label{oceanicdrag}
\end{align}
In the Coriolis term, in equation \eqref{EVP1}, we have used the notation $v^\perp = (-v_2, v_1)^\top$ for $v \in \mathbb{R}^2$. Like in the original paper \cite{hunke}, we have taken the mean ice thickness $h$ and the ice compactness $A$ to be constant, which in turn implies that the mass $m$ and the internal ice strength $P$ are also constants (as we have already assumed above). The given function $H_0$ in equation \eqref{EVP1} describes the ocean surface topography. All the remaining parameters in system \eqref{EVPsystem} and equations \eqref{atmosphericdrag}-\eqref{oceanicdrag} (in particular the angles $\phi$ and $\theta$) are introduced in Table \ref{notationtable}, below.

\begin{table}[h]
    \centering
\begin{tabular}{ |p{2cm}||p{5cm}|p{3cm}|p{3cm}|  }
\hline
Symbol & Meaning & Typical value & Equation of first appearance\\
\hline
$A$ & ice compactness (area covered by `thick' ice) & $0 \leq A \leq 1$ & \\
$c_a$ & air drag coefficient & $1.2 \cdot 10^{-3}$ & \eqref{atmosphericdrag} \\
$c_w$ & ocean drag coefficient & $5.5 \cdot 10^{-3}$ & \eqref{oceanicdrag} \\
$\mathcal{E}$ & elastic modulus & $0.25$ & \eqref{EVP2} \\
$g$ & gravitational constant & $9.81 \; \text{m\,s}^{-1}$ & \eqref{EVP1} \\
$H_0$ & sea surface height & & \eqref{EVP1} \\
$\Omega$ &  rotation parameter & $1.46 \cdot 10^{-4} \; \text{s}^{-1}$ & \eqref{EVP1} \\
$P$ & internal ice strength & & \eqref{EVP2} \\
$\phi$ & air turning angle & $25^\circ$ & \eqref{atmosphericdrag} \\
$\rho_a$ & air density & $1.3 \; \text{kg/m}^3$ & \eqref{atmosphericdrag} \\
$\rho_w$ & ocean water density & $1026 \; \text{kg/m}^3$ & \eqref{oceanicdrag} \\
$\theta$ & water turning angle & $25^\circ$ & \eqref{oceanicdrag} \\
$U_a$ & geostrophic wind & & \eqref{atmosphericdrag} \\
$U_w$ & geostrophic ocean current & & \eqref{oceanicdrag} \\
 \hline
\end{tabular}
    \caption{An overview of the notation used in the EVP model, the typical values are taken from: \cite{hunke,hunkelinearization,mehlmann,lemieux}.}
    \label{notationtable}
\end{table}

The EVP model was introduced in \cite{hunke} as a numerical regularisation of the viscous-plastic Hibler sea-ice model \cite{hibler}. The Hibler model corresponds to formally setting $\mathcal{E} = \infty$ in equations \eqref{EVP1} and \eqref{EVP2}. In a previous paper by the authors \cite{boutros2025}, it has been observed that the 1D EVP model is (formally) linearly ill-posed in Sobolev spaces (in the absence of the advection term and utilising a regularisation of the strain rate $\mathcal{D}$ of the type \eqref{strainrateregularisation}, below). This led to the introduction of the Voigt-regularisation $-\alpha^2 \partial_t \Delta \sigma$, with $\alpha > 0$, on the left-hand side of equation \eqref{EVP2} in \cite{boutros2025} in order to regularise the system. The motivation for this choice of regularisation is like that in the original EVP model \cite{hunke}, namely, it preserves the steady states as well as the formal asymptotic in time statistical (infinite-time average) solutions of the Hibler model. In addition, it is a modification of the elastic term (which was introduced for numerical purposes) rather than the viscous-plastic terms in the constitutive relation (for the stress tensor). 

In this work we consider the other case, namely introducing a Kelvin-Voigt regularisation $-\alpha^2 \partial_t \Delta u$, with $\alpha > 0$, on the left-hand side of equation \eqref{EVP1}. The motivation for doing so is twofold. On the one hand, this choice of regularisation is purely an introduction of a Kelvin-Voigt component into the rheology and therefore has a clear physical interpretation. On the other hand, from the mathematical analysis perspective the nonlinearities of the EVP model contain derivatives of $u$ rather than $\sigma$, which is a reason to introduce a regularisation for $u$ instead of $\sigma$. Due to the Kelvin-Voigt regularisation, we are able to prove the existence of solutions for initial data of lower regularity compared to the case of the Voigt-regularisation of the constitutive relation (for the stress tensor) in \cite{boutros2025}. In the present work we will establish two results, specifically, the local well-posedness of the EVP model with the advection term and Voigt-regularisation in the momentum balance \eqref{EVP1} in Theorem \ref{localwellposednessthm}, and in Theorem \ref{globalwellposednessthm} the global well-posedness of the EVP model without the advection term but including the Voigt regularisation in the momentum balance \eqref{EVP1}.
\subsection{Main results of this paper}
In this paper, we consider two different formulations of the EVP model with the Kelvin-Voigt regularisation to equation \eqref{EVP1}, namely both with and without the advection term in the momentum balance \eqref{EVP1}. The first formulation, which we will call the advective Kelvin-Voigt EVP model, is given by
\begin{subequations} \label{advectionsystem}
\begin{align}
	&\partial_t (u - \alpha^2 \Delta u) + u \cdot \nabla u  = \nabla \cdot \sigma + \mathcal{T}_a + \mathcal{T}_w + \Omega u^\perp - g \nabla H_0, \label{advection1} \\
	&\dfrac{1}{\mathcal E} \partial_t \sigma + \dfrac{4 \mathcal{D}}{P} (\sigma - \frac{1}{2} \Tr \sigma \mathbb I_2) + \dfrac{\mathcal{D}}{2P} \Tr \sigma \mathbb I_2 + \dfrac{\mathcal{D}}{2} \mathbb I_2  = D(u), \label{advection2}
	\\
	&u \lvert_{t = 0} = u_0, \quad \sigma \lvert_{t = 0} = \sigma_0, \label{advection3}
\end{align}
\end{subequations}
for some positive regularisation parameter $\alpha > 0$.

The second formulation that we study is the EVP model with the Voigt-regularisation and without the advection term, which we call the Kelvin-Voigt EVP model and is given by
\begin{subequations} \label{kelvinvoigtsystem}
\begin{align}
	&\partial_t (u - \alpha^2 \Delta u) = \nabla \cdot \sigma + \mathcal{T}_a + \mathcal{T}_w + \Omega u^\perp - g \nabla H_0, \label{kelvinvoigt1} \\
	&\dfrac{1}{\mathcal E} \partial_t \sigma + \dfrac{4 \mathcal{D}}{P} (\sigma - \frac{1}{2} \Tr \sigma \mathbb I_2) + \dfrac{\mathcal{D}}{2P} \Tr \sigma \mathbb I_2 + \dfrac{\mathcal{D}}{2} \mathbb I_2  = D(u), \label{kelvinvoigt2}
	\\
	&u \lvert_{t = 0} = u_0, \quad \sigma \lvert_{t = 0} = \sigma_0, \label{kelvinvoigt3}
\end{align}
\end{subequations}
for some positive regularisation parameter $\alpha > 0$.
We will prove in this paper that the advective Kelvin-Voigt EVP model \eqref{advectionsystem} is locally well-posed, which is stated in the following theorem.
\begin{theorem} \label{localwellposednessthm}
Let $u_0 \in H^2 (\mathbb{T}^2)$ and $\sigma_0 \in H^1 (\mathbb{T}^2) \cap L^\infty (\mathbb{T}^2)$, such that $\sigma_0 = \sigma_0^\top$ for a.e. $x \in \mathbb{T}^2$. Moreover, let $U_a, U_w \in L^4 ((0,T); L^4 (\mathbb{T}^2))$ and $H_0 \in L^2 ((0,T); H^1 (\mathbb{T}^2))$. Then there exists a time $T > 0$, which depends on $u_0$ and $\sigma_0$ (as well as $\alpha$, $c_a$, $c_w$, $\rho_a$, $\rho_w$, $\mathcal{E}$, $g$, $H_0$, $P$, $\phi$, $\theta$, $U_a$ and $U_w$), such that there exists a unique local-in-time strong solution $(u,\sigma)$ to the advective Kelvin-Voigt EVP model \eqref{advectionsystem}, which has the following regularity
\begin{equation}
u \in C([0,T];H^2 (\mathbb{T}^2)), \quad \sigma \in C([0,T]; H^1_w (\mathbb{T}^2)) \cap C([0,T]; L^2 (\mathbb{T}^2))\footnote{We recall the embedding $C([0,T]; H^1_w (\mathbb{T}^2)) \subset C([0,T]; L^2 (\mathbb{T}^2))$, but for the sake of clarity we have separately stated the inclusion of $\sigma$ in the space $C([0,T]; L^2 (\mathbb{T}^2))$ in Theorems \ref{localwellposednessthm} and \ref{globalwellposednessthm}. }.
\end{equation}
Moreover, the unique local solution $(u,\sigma)$ depends continuously on the initial data $(u_0, \sigma_0)$ (see estimate \eqref{continuousdependence}, below). In other words, the advective Kelvin-Voigt EVP model is locally well-posed.  
\end{theorem}
In the next theorem we will establish that the Kelvin-Voigt EVP model \eqref{kelvinvoigtsystem} is globally well-posed.
\begin{theorem} \label{globalwellposednessthm}
Let $u_0 \in H^2 (\mathbb{T}^2)$ and $\sigma_0 \in H^1 (\mathbb{T}^2) \cap L^\infty (\mathbb{T}^2)$, such that $\sigma_0 = \sigma_0^\top$ for a.e. $x \in \mathbb{T}^2$.  Moreover, let $U_a \in L^4 ((0,T); L^4 (\mathbb{T}^2))$, $U_w \in L^4 ((0,T); L^\infty (\mathbb{T}^2))$ and $H_0 \in L^2 ((0,T); H^1 (\mathbb{T}^2))$. Then there exists a unique global-in-time strong solution $(u,\sigma)$ to the Kelvin-Voigt EVP model \eqref{kelvinvoigtsystem}, so that for every $T > 0$
\begin{equation}
u \in C([0,T];H^2 (\mathbb{T}^2)), \quad \sigma \in C([0,T]; H^1_w (\mathbb{T}^2)) \cap C([0,T]; L^2 (\mathbb{T}^2)).
\end{equation}
The unique global solution $(u,\sigma)$ depends continuously on the initial data $(u_0, \sigma_0)$, and satisfies an analogous estimate to \eqref{continuousdependence}, below. Therefore, the Kelvin-Voigt EVP model is globally well-posed.
\end{theorem}
As an intermediate step in the proof of both of these theorems, we will utilise the following regularisation of the strain rate
\begin{equation} \label{strainrateregularisation}
\mathcal{D}_\epsilon \coloneqq \sqrt{\lvert D(u) \rvert^2 + \epsilon^2},
\end{equation}
for a small parameter $\epsilon > 0$, which we will eventually send to zero in the proofs. Therefore, in the presence of the advection term, we will consider the following intermediate system, which we will refer to as the regularised advective Kelvin-Voigt EVP model,
\begin{subequations} \label{advectionsystemepsilon}
\begin{align}
&\partial_t (u - \alpha^2 \Delta u) + u \cdot \nabla u  = \nabla \cdot \sigma + \mathcal{T}_a + \mathcal{T}_w + \Omega u^\perp - g \nabla H_0, \label{advection1epsilon} \\
&\dfrac{1}{\mathcal E} \partial_t \sigma + \dfrac{4 \mathcal{D}_\epsilon}{P} (\sigma - \frac{1}{2} \Tr \sigma \mathbb I_2) + \dfrac{\mathcal{D}_\epsilon}{2P} \Tr \sigma \mathbb I_2 + \dfrac{\mathcal{D}_\epsilon}{2} \mathbb I_2  = D(u). \label{advection2epsilon}
\\
&u \lvert_{t = 0} = u_0, \quad \sigma \lvert_{t = 0} = \sigma_0. \label{advection3epsilon}
\end{align}
\end{subequations}
In the case without the advection term, we will analyse the following intermediate system
\begin{subequations} \label{kelvinvoigtsystemepsilon}
\begin{align}
&\partial_t (u - \alpha^2 \Delta u) = \nabla \cdot \sigma + \mathcal{T}_a + \mathcal{T}_w + \Omega u^\perp - g \nabla H_0, \label{kelvinvoigt1epsilon} \\
&\dfrac{1}{\mathcal E} \partial_t \sigma + \dfrac{4 \mathcal{D}_\epsilon}{P} (\sigma - \frac{1}{2} \Tr \sigma \mathbb I_2) + \dfrac{\mathcal{D}_\epsilon}{2P} \Tr \sigma \mathbb I_2 + \dfrac{\mathcal{D}_\epsilon}{2} \mathbb I_2  = D(u), \label{kelvinvoigt2epsilon}
	\\
&u \lvert_{t = 0} = u_0, \quad \sigma \lvert_{t = 0} = \sigma_0, \label{kelvinvoigt3epsilon}
\end{align}
\end{subequations}
which we will refer to as the regularised Kelvin-Voigt EVP model. In the proofs of Theorems \ref{localwellposednessthm} and \ref{globalwellposednessthm} we will first show the existence of strong solutions to the intermediate systems \eqref{advectionsystemepsilon} and \eqref{kelvinvoigtsystemepsilon}, respectively, and then pass to the limit $\epsilon \rightarrow 0$ in order to demonstrate the well-posedness of the limiting (advective) Kelvin-Voigt systems \eqref{advectionsystem} and \eqref{kelvinvoigtsystem}, respectively.
\begin{remark}
To the best of our knowledge, this is the first rigorous result in the analysis of sea-ice dynamics on the limit $\epsilon \rightarrow 0$ (i.e., the removal of the cutoff of the viscosity coefficients) in the presence of the advection term in the momentum balance in the multi-dimensional case. The removal of the viscosity cutoff has already been justified rigorously in \cite{boutros2025} for the two-dimensional Voigt-EVP model, but the estimates do not allow the passage to the limit $\epsilon \rightarrow 0$ if the advection term is included in the momentum balance. We note that for the multi-dimensional Hibler model (i.e., including advection) the rigorous treatment of the limit $\epsilon \rightarrow 0$ remains a major open problem, cf. \cite{liu}. In the case of the one-dimensional Hibler model, the limit $\epsilon \rightarrow 0$ has been recently treated in \cite{liu2026}.
\end{remark}
\begin{remark}
We also note that, as opposed to the results in \cite[Theorem 1.2]{boutros2025}, in Theorem \ref{globalwellposednessthm} we are not required to make an assumption on the value of the water turning angle $\theta$. This is due to the fact that the velocity is regularised instead of the stress tensor, which means that the higher order estimate can be closed without size restrictions on the turning angle.
\end{remark}

\subsection{Overview of the literature}
\subsubsection{The sea-ice modelling literature}
Because of the complex nature of sea ice as a multiscale material, modelling its rheology and dynamics accurately remains a major challenge. Due to the vast scope of the sea-ice dynamics literature and constraints of brevity, the overview that will be given below will unfortunately be limited. We refer to \cite{boutros2025} for more references on this topic. 

The Arctic Ice Dynamics Joint Experiment led to the development of an elastic-plastic rheology to model sea-ice dynamics \cite{coon}. Several years later, in \cite{hibler} a viscous-plastic rheology was developed by Hibler to model sea-ice drift, which remains one of the most standard sea-ice rheologies in use today. Although the Hibler model has been successful in describing sea-ice drift \cite{losch}, its computational implementation remains very expensive (particular so when using explicit numerical schemes) \cite{bouillon2009,ip,koldunov}. Over time, implicit numerical schemes have been developed to compute solutions to the Hibler model, for example the Jacobian-free Newton Krylov solver \cite{lemieux,losch2014,seinen} and the line relaxation method \cite{zhang}.

In order to improve the computational efficiency of the Hibler model, in \cite{hunke} the elastic-viscous-plastic (EVP) sea-ice model was introduced, which is an (elastic) relaxation of the Hibler model. The advantage of the EVP approach is that it facilitates the use of explicit numerical schemes (and hence also of parallel computing) and therefore its use in numerical simulations is less computationally intensive \cite{hunke2002,hunkelinearization,bouillon2009}. Progressively, several reformulations of the EVP model were introduced to improve its modelling and computational performance \cite{hunkelinearization,bouillon}. We remark also that several different sea-ice rheologies have been proposed, see for example \cite{dansereau,tsamados,wilchinsky,heorton2018}. For the multi-scale analysis of sea-ice dynamics we refer for example to \cite{chen,deng,toppaladoddi}. 

\subsubsection{The mathematical analysis literature on sea-ice dynamics}
Until recently, the mathematical analysis of sea-ice models had received very little attention in the literature. The paper \cite{liu} proved the local well-posedness of the Hibler model, in which the authors slightly adapted the regularisation of the strain rate from the original paper \cite{hibler} (essentially by replacing the cutoff of the strain rate with the maximum function by a regularisation of the type $\mathcal{D}_\epsilon$ from \eqref{strainrateregularisation}, as is also used in this paper, cf. \cite{kreyscher}). A proof of the same result, using different techniques, was later given in \cite{brandt2025}. The removal of the cutoff of the strain rate was studied recently for the 1D momentum equation of the Hibler model in \cite{liu2026}. In particular, the existence of a BV weak solution for the 1D momentum part of the Hibler model was established.

At the same time as \cite{liu}, the work \cite{brandt} studied a modified version of the Hibler model, in which the same momentum equation was used, but the mean ice thickness $h$ and the ice compactness $A$ are further regularised by adding diffusion (which lacks physical justification). Such models were studied in follow-up works \cite{binz2022,binz2024,brandt2023}. We also mention here, that the Hibler model with a different regularisation of the strain rate (by using the hyperbolic tangent) has been studied in \cite{chatta} (and see also \cite{chatta2025}). We note that in \cite{piersanti} a model for shallow (land) ice-sheets was analysed.

The first well-posedness results on the EVP model were obtained in the aforementioned paper \cite{boutros2025}, in which the global well-posedness of the EVP model was established in the case of a Voigt regularisation of the constitutive stress-tensor relation (i.e., the evolution equation for the stress tensor). This was the first global existence result for a sea-ice model for the case of large initial data. Moreover, in \cite{boutros2025} a rigorous treatment of the limit $\epsilon \rightarrow 0$ in the case of two dimensions was given. In \cite{boutros2025} (following \cite{hunke}), the advection term was ignored in the momentum balance (as it is essentially lower order), which is a very common assumption in sea-ice modelling (cf. \cite{lepparanta}). Using this approximation from \cite{hunke,boutros2025}, global existence results for the momentum equation of the Hibler model without the advection term were obtained in \cite{dingel,denk}. The results in \cite{dingel,denk} can be considered refinements of existing works on the total variation flow, see for example \cite{giga2001,giga2010,andreu2001,andreu2001minimizing,bellettini}. In \cite{mehlmann}, a formal $H^1$-estimate was derived for the revised EVP model (under restrictive assumptions), which was used to verify the consistency of numerical schemes for this model. 

Finally, we also remark that the analysis of the EVP model shares some similar features with the analysis of the Oldroyd-B model for non-Newtonian flows. The Oldroyd-B model has for example been studied in \cite{chemin,constantin2012,constantin2021,elgindi,guillope,kupferman,lin,lions}. The (Kelvin-)Voigt type regularisation has been used in several other contexts, for instance in the analysis of the Navier-Stokes equations and turbulence modelling (and also as a method to study steady states), see for example \cite{oskolkov1977,oskolkov1997,larios2010,larios2010boussinesq,larios2014,antontsev,constantin2023,ignatova}. Equations with the Voigt regularisation often belong to the class of pseudo-parabolic equations, which have been studied in \cite{showalter1970a,showalter1970b,showalter1975}.

\section{Preliminaries}
In this work, $\varphi : \mathbb{R}^2 \rightarrow \mathbb{R}$ will denote a standard nonnegative radial $C^\infty$ mollifier with compact support such that $\int_{\mathbb{R}^2} \varphi \dx = 1$. Moreover, for a parameter $\delta > 0$ we will define
\begin{equation*}
\varphi_\delta (x) \coloneqq \frac{1}{\delta^2} \varphi \left( \frac{x}{\delta} \right).
\end{equation*}
In this paper we will use the following notational convention
\begin{equation} \label{mollification}
f^\delta \coloneqq f * \varphi_\delta.
\end{equation}
We will use the notation $A \lesssim B$ to mean that there exists a constant $C$ such that $A \leq C B$. In general, throughout this paper the constant $C$ will not depend on the parameters $\beta$ (cf. system \eqref{regularisedkelvinvoigtsystem}) and $\epsilon$ that we will eventually send to zero in the proof (unless we will indicate otherwise). In order to estimate the contribution from the strain rates, we will use the following lemma, which was proved in \cite[Lemma 2.3]{boutros2025}.
\begin{lemma} \label{strainratelemma}
For $v_1, v_2 \in H^1 (\mathbb{T}^2)$ the following estimate holds
\begin{equation}
\left\lVert \sqrt{\lvert D(v_1) \rvert^2 + \epsilon^2} - \sqrt{\lvert D(v_2) \rvert^2 + \epsilon^2} \right\rVert_{L^2} \leq \left\lVert \nabla v_1 - \nabla v_2 \right\rVert_{L^2},
\end{equation}
for any $\epsilon \geq 0$.
\end{lemma}
Throughout this paper, we will assume that the stress tensor $\sigma$ is symmetric. This assumption is made on the initial datum $\sigma_0$ and is preserved under the evolution, as we will state in the next lemma.
\begin{lemma}[Invariance of the symmetry of the stress tensor]
Let $(u,\sigma)$ be a solution to the advective Kelvin-Voigt EVP model \eqref{advectionsystem} or the Kelvin-Voigt EVP model \eqref{kelvinvoigtsystem} such that $u \in C([0,T];H^2 (\mathbb{T}^2))$ and $\sigma \in C([0,T]; H^1_w (\mathbb{T}^2)) \cap C([0,T]; L^2 (\mathbb{T}^2))$. Assume that $\sigma_0 \in H^1 (\mathbb{T}^2)$ is symmetric almost everywhere in $\mathbb{T}^2$, then $\sigma (x,t)$ is symmetric for all $t \in [0,T]$ and almost every $x \in \mathbb{T}^2$.
\end{lemma}
\begin{proof}
Let $A(\sigma)$ denote the antisymmetric part of the stress tensor $\sigma$, i.e.,
\begin{equation}
A (\sigma) \coloneqq \frac{1}{2} (\sigma - \sigma^\top).
\end{equation}
It is clear that $A(\sigma)$ has the same regularity as $\sigma$. It follows from equation \eqref{advection2} or \eqref{kelvinvoigt2} that $A(\sigma)$ satisfies the following equation
\begin{equation}
\dfrac{1}{\mathcal E} \partial_t A(\sigma) + \dfrac{4 \mathcal{D}}{P} A(\sigma)  = 0.
\end{equation}
One can check that this equation holds in $C ([0,T]; L^2 (\mathbb{T}^2))$. Therefore we can take the $L^2 (\mathbb{T}^2)$-inner product with $A(\sigma)$, which gives (after integrating in time) 
\begin{equation}
\lVert A(\sigma) (\cdot, t) \rVert_{L^2}^2 \leq \lVert A(\sigma_0) \rVert_{L^2}^2 = 0,
\end{equation}
since $\mathcal{E} > 0$ and $\dfrac{4 \mathcal{D}}{P} \geq 0$. Consequently, the symmetry of the stress tensor is preserved by the evolution of the (advective) Kelvin-Voigt EVP model. A similar argument also applies for the regularised (advective) Kelvin-Voigt EVP model.
\end{proof}
In the proof of both Theorem \ref{localwellposednessthm} and \ref{globalwellposednessthm}, it will be crucial to use an $L^\infty (\mathbb{T}^2)$-estimate for the stress tensor $\sigma$, which we establish in the next lemma. 
\begin{lemma} \label{Lpestimatelemma}
Let $u \in C([0,T];H^1 (\mathbb{T}^2))$ and $\sigma \in C([0,T];H^1_w (\mathbb{T}^2)) \cap C([0,T]; L^2 (\mathbb{T}^2))$, with $\sigma_0 \in L^\infty (\mathbb{T}^2) \cap H^1 (\mathbb{T}^2)$, be two functions which satisfy the following equation (where $P > 0$ is a given constant)
\begin{equation}
\dfrac{1}{\mathcal E} \partial_t \sigma + \dfrac{4 \mathcal{D}_\epsilon}{P} (\sigma - \frac{1}{2} \Tr \sigma \mathbb I_2) + \dfrac{\mathcal{D}_\epsilon}{2P} \Tr \sigma \mathbb I_2 + \dfrac{\mathcal{D}_\epsilon}{2} \mathbb I_2  = D(u),
\end{equation}
which is the same as \eqref{advection2epsilon} and \eqref{kelvinvoigt2epsilon}. Then it holds that $\sigma \in L^\infty ((0,T); L^\infty (\mathbb{T}^2))$ and we have
\begin{equation}
\lVert \sigma (\cdot, t) \rVert_{L^\infty} \leq \lVert \sigma_0 \rVert_{L^\infty} + 2 P.
\end{equation}
\end{lemma}
\begin{proof}
We first introduce the new unknown
\begin{equation} \label{newunkown}
\tau \coloneqq \sigma + \frac{P}{2} \mathbb I_2,
\end{equation}
which satisfies the following equation
\begin{equation} \label{modifiedstressequation}
\dfrac{1}{\mathcal E} \partial_t \tau + \dfrac{4 \mathcal{D}_\epsilon}{P} (\tau - \frac{1}{2} \Tr \tau \mathbb I_2) + \dfrac{\mathcal{D}_\epsilon}{2P} \Tr \tau \mathbb I_2 = D(u).
\end{equation}
From the regularity of $u$ and $\tau$ it follows that $\partial_t \tau \in C([0,T]; L^{2 - \gamma} (\mathbb{T}^2))$ for any $\gamma \in (0,1)$. Therefore, for every $p > 3$, we can take the duality pairing of equation \eqref{modifiedstressequation} with $P^{-p+1} \tau \lvert \tau \rvert^{p-2}$, which gives
\begin{align*}
&\frac{P}{p \mathcal{E}} \frac{\textrm{d}}{\dt} \lVert \tau / P \rVert_{L^p}^p + \int_{\mathbb{T}^2} \bigg[ 4 \mathcal{D}_\epsilon \left( \frac{1}{P} \right)^p \bigg\lvert \tau - \frac{1}{2} \Tr \tau \mathbb I_2 \bigg\rvert^2 \lvert \tau \rvert^{p-2} + \dfrac{\mathcal{D}_\epsilon}{2} \left( \frac{1}{P} \right)^p \lvert \Tr \tau \rvert^2 \lvert \tau \rvert^{p-2} \bigg] \dx \\
&\leq \int_{\mathbb{T}^2} P^{-p+1} \mathcal{D}_\epsilon \lvert \tau \rvert^{p-1} \dx = \int_{\mathbb{T}^2} \mathcal{D}_\epsilon^{1/p} \cdot P^{-(p-1)} \mathcal{D}_\epsilon^{(p-1)/p} \lvert \tau \rvert^{p-1} \dx \\
&\leq \int_{\mathbb{T}^2} \bigg[\frac{1}{p} \mathcal{D}_\epsilon + \frac{p-1}{p} \mathcal{D}_\epsilon \left( \frac{1}{P} \right)^p \lvert \tau \rvert^p \bigg] \dx.
\end{align*}
Now, we recall the property
\begin{equation*}
\lvert \tau \rvert^2 = \bigg\lvert \tau - \frac{1}{2} \Tr \tau \mathbb I_2 \bigg\rvert^2 + \frac{1}{2} \lvert \Tr \tau \rvert^2.
\end{equation*}
Therefore, we can deduce the following estimate
\begin{align*}
\frac{P}{p \mathcal{E}} \frac{\textrm{d}}{\dt} \lVert \tau / P \rVert_{L^p}^p &+ \int_{\mathbb{T}^2} \bigg[ 3 \mathcal{D}_\epsilon \left( \frac{1}{P} \right)^p \bigg\lvert \tau - \frac{1}{2} \Tr \tau \mathbb I_2 \bigg\rvert^2 \lvert \tau \rvert^{p-2} + \dfrac{\mathcal{D}_\epsilon}{p} \left( \frac{1}{P} \right)^p \lvert \tau \rvert^p \bigg] \dx \leq \frac{1}{p} \int_{\mathbb{T}^2} \mathcal{D}_\epsilon \dx.
\end{align*}
Dropping the coercive terms (the positive terms on the left-hand side) and integrating in time leads to
\begin{align}
\dfrac{1}{\mathcal E} \lVert (\tau / P) (\cdot, t) \rVert_{L^p}^p \leq \dfrac{1}{\mathcal E} \lVert \tau_0 / P \rVert_{L^p}^p + \frac{1}{P} \int_0^t \int_{\mathbb{T}^2} \mathcal{D}_\epsilon \dx \dt' \leq \dfrac{1}{\mathcal E} \lVert \tau_0 / P \rVert_{L^p}^p + \frac{1}{P} \int_0^t (\lVert \nabla u \rVert_{L^2} + \epsilon) \dt' \label{integratedLpestimate}.
\end{align}
Now, taking the $p$-th root of equation \eqref{integratedLpestimate} and using the subadditivity property of the $p$-th root (i.e. $(x + y)^{1/p} \leq x^{1/p} + y^{1/p}$ for $x,y \geq 0$, $ p \geq 1 $) we have
\begin{equation} \label{integratedLpestimate2}
\dfrac{1}{\mathcal{E}^{1/p}} \lVert (\tau / P) (\cdot, t) \rVert_{L^p} \leq \dfrac{1}{\mathcal{E}^{1/p}} \lVert \tau_0 / P \rVert_{L^p} + \left( \frac{1}{P} \int_0^t (\lVert \nabla u \rVert_{L^2} + \epsilon) \dt' \right)^{1/p},
\end{equation}
which holds for every $p > 3$. Next, we recall the following result for finite-measure spaces: If $f \in L^p (\mathbb{T}^2)$ for any $p \in [p_0,\infty)$, for some $p_0 \geq 1$, and $\lVert f \rVert_{L^p} \leq K$ (for a constant $K$ independent of $p$), then $f \in L^\infty (\mathbb{T}^2)$ and $\lVert f \rVert_{L^\infty} \leq K$. Moreover, we have
\begin{equation} \label{Linftylimit}
\lVert f \rVert_{L^\infty} = \lim_{p \rightarrow \infty} \lVert f \rVert_{L^p},
\end{equation}
the proof of which can be found for example in \cite[Theorem 2.8]{adams}. 
Now sending $p \rightarrow \infty$ in equation \eqref{integratedLpestimate2} and using property \eqref{Linftylimit} we find
\begin{equation}
\lVert \tau (\cdot, t) \rVert_{L^\infty} \leq \lVert \tau_0 \rVert_{L^\infty} + P.
\end{equation}
\end{proof}
\section{Proof of Theorem \ref{localwellposednessthm}: The advective Kelvin-Voigt EVP model \eqref{advectionsystem}}
\subsection{Setup of the Galerkin approximation scheme}
As an intermediate step in the proof, we will use the regularisation of the strain rate given in equation \eqref{strainrateregularisation}. Therefore we will first construct a local strong solution to the following further regularised advective Kelvin-Voigt EVP model in $\mathbb{T}^2$ (for some parameters $\beta > 0$ and $\delta \in \left(0, \frac{1}{2}\right)$)
\begin{subequations} \label{regularisedkelvinvoigtsystem}
\begin{align}
	&\partial_t (u - \alpha^2 \Delta u + \beta^4 \Delta^2 u) + u \cdot \nabla u = \nabla \cdot \sigma + \mathcal{T}_a + \mathcal{T}_w + \Omega u^\perp - g \nabla H_0, \label{regularisedkelvinvoigt1} \\
	&\dfrac{1}{\mathcal E} \partial_t \sigma + \dfrac{4 \mathcal{D}_\epsilon}{P} (\sigma - \frac{1}{2} \Tr \sigma \mathbb I_2) + \dfrac{\mathcal{D}_\epsilon}{2P} \Tr \sigma \mathbb I_2 + \dfrac{\mathcal{D}_\epsilon}{2} \mathbb I_2  = D(u), \label{regularisedkelvinvoigt2}
	\\
	&u \lvert_{t = 0} = u_0^\delta, \quad \sigma \lvert_{t = 0} = \sigma_0^\delta, \label{regularisedkelvinvoigt3}
\end{align}
\end{subequations}
where the mollified initial data $u_0^\delta$ and $\sigma_0^\delta$ have been obtained from $u_0$ and $\sigma_0$ according to \eqref{mollification}. We will construct a solution to \eqref{regularisedkelvinvoigtsystem} by means of the Galerkin method. Note that we have introduced a regularisation term $\beta^4 \partial_t \Delta^2 u$ in the momentum balance in order to have sufficient regularity bounds in the interesting of constructing an approximate solution, to which we will then apply Lemma \ref{Lpestimatelemma} (as the $L^\infty$-bound from Lemma \ref{Lpestimatelemma} is not compatible with the Galerkin approximations). Using Lemma \ref{Lpestimatelemma} will then yield regularity bounds which are independent of $\beta > 0$. The introduction of the double regularisation term $\beta^4 \partial_t \Delta^2 u$ means also that we need to mollify the initial data, as we have done above. We consider a solution of the following form
\begin{equation*}
	u_N(x,t) = \sum_{k \in \mathbb{Z}^2, \lvert k \rvert \leq N} a_k (t) e^{2 \pi i k \cdot x}, \quad \sigma_N(x,t) = \sum_{k \in \mathbb{Z}^2, \lvert k \rvert \leq N} b_k(t) e^{2 \pi i k \cdot x},
\end{equation*}
which solves the Galerkin system of order $N$
\begin{align} 
	&\partial_t (u_N - \alpha^2 \Delta u_N + \beta^4 \Delta^2 u_N) + \mathbf{P}_N (u_N \cdot \nabla u_N)  \nonumber \\
    &= \nabla \cdot \sigma_N + \mathbf{P}_N \mathcal{T}_a + \mathbf{P}_N \mathcal{T}_w^N + \Omega u^\perp_N - g \mathbf{P}_N \nabla H_0, \label{galerkin1} \\ 
	&\frac{1}{\mathcal{E}} \partial_t \sigma_N + \mathbf{Q}_N \bigg[\frac{4 \mathcal{D}_\epsilon^N}{P} (\sigma_N - \frac{1}{2} \Tr \sigma_N \mathbb I_2) \bigg] + \mathbf{Q}_N \bigg[ \frac{\mathcal{D}_\epsilon^N}{2 P} \Tr \sigma_N \mathbb I_2 \bigg] + \frac{\mathbf{Q}_N \mathcal{D}_\epsilon^N}{2} \mathbb I_2 = D(u_N), 
	\label{galerkin2} \\
	&u_N \lvert_{t = 0} = \mathbf{P}_N u_0^\delta, \quad \sigma_N \lvert_{t = 0} = \mathbf{Q}_N \sigma_0^\delta, \label{galerkin3}
\end{align}
where we have introduced the following notation
\begin{align}
	\mathcal{T}_{w}^N &= c_w \rho_w \big\lvert U_w - u_N \big\rvert \bigg[ (U_w - u_N) \cos \theta + (U_w - u_N)^\perp \sin \theta \bigg] \label{galerkindrag} \\
	\mathcal{D}_\epsilon^N &= \sqrt{\lvert D(u_N) \rvert^2 + \epsilon^2 }.
\end{align}
In these equations, $\mathbf{P}_N$ is the $L^2 (\mathbb{T}^2)$-projection of two-dimensional vector fields onto their Fourier modes up to order $N$, the map $\mathbf{Q}_N$ is the $L^2 (\mathbb{T}^2)$-projection of $2 \times 2$ symmetric matrix fields onto their Fourier modes up to order $N$. Note that by the Picard-Lindelöf theorem, the Galerkin ODE system \eqref{galerkin1}-\eqref{galerkin3} has a local-in-time solution. 
\subsection{Construction of a solution for the regularised system \eqref{regularisedkelvinvoigtsystem}} \label{regularisedsystemsection}
We will now derive the a priori estimates (which are uniform in $N$ and $\epsilon$) that will be used to construct the solution of the system \eqref{regularisedkelvinvoigtsystem}, after which we will send $\beta, \delta \rightarrow 0$ in Section \ref{limitsection} (and then $\epsilon \rightarrow 0$ after that). Taking the $L^2 (\mathbb{T}^2)$ inner products of equations \eqref{galerkin1} and \eqref{galerkin2} with $u_N$ and $-\Delta u_N$, and respectively with $\sigma_N$ and $-\Delta \sigma_N$, and adding the resultants leads to
\begin{align*}
&\frac{1}{2} \frac{\textrm{d}}{\dt} \bigg[ \lVert u_N \rVert_{H^1}^2 + \mathcal{E}^{-1} \lVert \sigma_N \rVert_{H^1}^2 + \alpha^2 \lVert \nabla u_N \rVert_{H^1}^2 + \beta^4 \lVert \Delta u_N \rVert_{H^1}^2 \bigg] \\
&= \int_{\mathbb{T}^2} \bigg[ (\nabla \cdot \sigma_N) \cdot (u_N - \Delta u_N) + D(u_N) : (\sigma_N - \Delta \sigma_N) \bigg] \dx \\
&- \int_{\mathbb{T}^2} \big[(u_N \cdot \nabla) u_N \big] \cdot \big[u_N - \Delta u_N \big] \dx \\
&+ \int_{\mathbb{T}^2} \bigg[ \mathbf{P}_N \mathcal{T}_a + \mathbf{P}_N \mathcal{T}_w^N + \Omega u^\perp_N - g \mathbf{P}_N \nabla H_0 \big] \cdot \big[u_N - \Delta u_N \big] \dx \\
&- \int_{\mathbb{T}^2} \bigg[ \frac{4 \mathcal{D}_\epsilon^N}{P} (\sigma_N - \frac{1}{2} \Tr \sigma_N \mathbb I_2) + \frac{\mathcal{D}_\epsilon^N}{2 P} \Tr \sigma_N \mathbb I_2 + \frac{\mathcal{D}_\epsilon^N}{2} \mathbb I_2 \bigg] \cdot \big(\sigma_N - \Delta \sigma_N \big) \dx \\
&\eqqcolon I_1 + I_2 + I_3 + I_4.
\end{align*}
Now we treat the various contributions $I_1, \ldots, I_4$ separately. We have
\begin{align*}
I_1 &= \int_{\mathbb{T}^2} \bigg[(\nabla \cdot \sigma_N) \cdot u_N + \sigma_N : D(u_N) \bigg] \dx + \sum_{i=1}^2 \int_{\mathbb{T}^2} \bigg[ (\nabla \cdot \partial_i \sigma_N) \cdot \partial_i u_N + D(\partial_i u_N) : \partial_i \sigma_N \bigg] \dx \\
&= 0,
\end{align*}
which follows from the divergence theorem. We estimate the contribution $I_2$ from the advection term as follows
\begin{align*}
\lvert I_2 \rvert &\lesssim \lVert u_N \rVert_{L^4}^2 \lVert \nabla u_N \rVert_{L^2} + \lVert u_N \rVert_{L^4} \lVert \nabla u_N \rVert_{L^4} \lVert \Delta u_N \rVert_{L^2} \lesssim \lVert u_N \rVert_{H^2}^3, 
\end{align*}
where we have used the Sobolev embedding theorem. Next, we estimate the contribution $I_3$ from the drag forces, i.e., in view of \eqref{atmosphericdrag} and \eqref{galerkindrag}, we have
\begin{align*}
\lvert I_3 \rvert &\lesssim \lVert u_N \rVert_{H^2} \bigg[ \lVert U_a \rVert_{L^4}^2 + \lVert U_w \rVert_{L^4}^2 + \lVert u_N \rVert_{L^4}^2 + \lVert H_0 \rVert_{H^1} \bigg].
\end{align*}
In order to treat the remaining term $I_4$, we decompose this term in the following manner
\begin{align*}
&- \int_{\mathbb{T}^2} \bigg[ \frac{4 \mathcal{D}_\epsilon^N}{P} (\sigma_N - \frac{1}{2} \Tr \sigma_N \mathbb I_2) + \frac{\mathcal{D}_\epsilon^N}{2 P} \Tr \sigma_N \mathbb I_2 + \frac{\mathcal{D}_\epsilon^N}{2} \mathbb I_2 \bigg] \cdot \big(\sigma_N - \Delta \sigma_N \big) \dx \\
&= - \int_{\mathbb{T}^2} \bigg[ \frac{4 \mathcal{D}_\epsilon^N}{P} (\sigma_N - \frac{1}{2} \Tr \sigma_N \mathbb I_2) + \frac{\mathcal{D}_\epsilon^N}{2 P} \Tr \sigma_N \mathbb I_2 + \frac{\mathcal{D}_\epsilon^N}{2} \mathbb I_2 \bigg] \cdot \sigma_N \dx \\
&+ \int_{\mathbb{T}^2} \bigg[ \frac{4 \mathcal{D}_\epsilon^N}{P} (\sigma_N - \frac{1}{2} \Tr \sigma_N \mathbb I_2) + \frac{\mathcal{D}_\epsilon^N}{2 P} \Tr \sigma_N \mathbb I_2 + \frac{\mathcal{D}_\epsilon^N}{2} \mathbb I_2 \bigg] \cdot \Delta \sigma_N \dx \eqqcolon I_{41} + I_{42}.
\end{align*}
We use the Cauchy-Schwarz inequality to obtain
\begin{align*}
\lvert I_{41} \rvert &= \bigg\lvert \int_{\mathbb{T}^2} \bigg[ \frac{4 \mathcal{D}_\epsilon^N}{P} (\sigma_N - \frac{1}{2} \Tr \sigma_N \mathbb I_2) + \frac{\mathcal{D}_\epsilon^N}{2 P} \Tr \sigma_N \mathbb I_2 + \frac{\mathcal{D}_\epsilon^N}{2} \mathbb I_2 \bigg] \cdot \sigma_N \dx \bigg\rvert \\
&\lesssim \lVert \mathcal{D}_\epsilon^N \rVert_{L^2} (1 + \lVert \sigma_N \rVert_{L^4}^2) \lesssim (\lVert \nabla u_N \rVert_{L^2} + 1) (1 + \lVert \sigma_N \rVert_{L^4}^2).
\end{align*}
In order to estimate $I_{42}$ (which will involve $\nabla \mathcal{D}_\epsilon^N$), we notice that
\begin{equation*}
\nabla \mathcal{D}_\epsilon^N = \frac{D(u_N) : \big[ \nabla D(u_N) \big]}{\mathcal{D}_\epsilon^N}.
\end{equation*}
This then leads to the following estimates
\begin{equation} \label{strainrateestimates}
\lVert \nabla \mathcal{D}_\epsilon^N \rVert_{L^2} \lesssim \lVert \Delta u_N \rVert_{L^2}, \quad \lVert \nabla \mathcal{D}_\epsilon^N \rVert_{L^4} \lesssim \lVert \Delta u_N \rVert_{L^4},
\end{equation}
where the constants in the right-hand sides of \eqref{strainrateestimates} are independent of $\epsilon$ and $N$. We then deduce that
\begin{align*}
\lvert I_{42} \rvert &= \bigg\lvert \int_{\mathbb{T}^2} \bigg[ \frac{4 \mathcal{D}_\epsilon^N}{P} (\sigma_N - \frac{1}{2} \Tr \sigma_N \mathbb I_2) + \frac{\mathcal{D}_\epsilon^N}{2 P} \Tr \sigma_N \mathbb I_2 + \frac{\mathcal{D}_\epsilon^N}{2} \mathbb I_2 \bigg] \cdot \Delta \sigma_N \dx \bigg\rvert \\
&= \bigg\lvert \int_{\mathbb{T}^2} \bigg[ \frac{4 \nabla \mathcal{D}_\epsilon^N}{P} (\sigma_N - \frac{1}{2} \Tr \sigma_N \mathbb I_2) + \frac{\nabla \mathcal{D}_\epsilon^N}{2 P} \Tr \sigma_N \mathbb I_2 + \frac{4 \mathcal{D}_\epsilon^N}{P} \nabla (\sigma_N - \frac{1}{2} \Tr \sigma_N \mathbb I_2) \\
&+ \frac{\mathcal{D}_\epsilon^N}{2 P} \nabla \Tr \sigma_N \mathbb I_2 + \frac{\nabla \mathcal{D}_\epsilon^N}{2} \mathbb I_2 \bigg] \cdot \nabla \sigma_N \dx \bigg\rvert \\
&\lesssim \lVert \Delta u_N \rVert_{L^4} \lVert \nabla \sigma_N \rVert_{L^2} \lVert \sigma_N \rVert_{L^4} + \lVert \nabla u_N \rVert_{L^\infty} \lVert \nabla \sigma_N \rVert_{L^2}^2 + \lVert \Delta u_N \rVert_{L^2} \lVert \nabla \sigma_N \rVert_{L^2} \\
&\lesssim \lVert u_N \rVert_{H^3} \lVert \sigma_N \rVert_{H^1}^2 + \lVert u_N \rVert_{H^2} \lVert \sigma_N \rVert_{H^1}.
\end{align*}
By combining the estimates on $I_1, I_2, I_3$ and $I_4$ (i.e. $I_{41}$ and $I_{42}$), we therefore find that
\begin{align*}
&\frac{1}{2} \frac{\textrm{d}}{\dt} \bigg[ \lVert u_N \rVert_{H^1}^2 + \mathcal{E}^{-1} \lVert \sigma_N \rVert_{H^1}^2 + \alpha^2 \lVert \nabla u_N \rVert_{H^1}^2 + \beta^4 \lVert \Delta u_N \rVert_{H^1}^2 \bigg] \lesssim \big[ \lVert u_N \rVert_{H^3}^2 + \lVert \sigma_N \rVert_{H^1}^2 + 1 \big]^{3/2},
\end{align*}
where we have used Young's inequality. Therefore, we deduce that there exists a time $T > 0$ and some constant $K \in (0, \infty)$ such that
\begin{align*}
\lVert u_N \rVert_{L^\infty ((0,T); H^3 (\mathbb{T}^2))} &+ \lVert \sigma_N \rVert_{L^\infty ((0,T); H^1 (\mathbb{T}^2))} + \lVert \partial_t u_N \rVert_{L^\infty ((0,T); H^4 (\mathbb{T}^2))} \\
&+ \lVert \partial_t \sigma_N \rVert_{L^\infty ((0,T); H^1 (\mathbb{T}^2))} \leq K < \infty.
\end{align*}
Note that both $T$ and $K$ are independent of $N$ and $\epsilon$, but they do depend on $\alpha$, $\beta$, $\delta$, $u_0$, $\sigma_0$, $c_a$, $c_w$, $\rho_a$, $\rho_w$, $\mathcal{E}$, $g$, $H_0$, $P$, $\phi$, $\theta$, $U_a$ and $U_w$.

Therefore, by applying the Banach-Alaoglu and Aubin-Lions compactness theorems, we deduce that there exists a limit $(u,\sigma) \in C([0,T];H^3 (\mathbb{T}^2)) \times C([0,T];H^1(\mathbb{T}^2))$ for which we have the following convergence results as $N \rightarrow \infty$ (by passing to a subsequence, if necessary)
\begin{align}
    u_N & \overset{\ast}{\rightharpoonup} u &  \text{weakly-$* $ in}& ~ L^\infty ((0,T);H^3 (\mathbb{T}^2)), \\
    \partial_t u_N & \overset{\ast}{\rightharpoonup} \partial_t u &  \text{weakly-$* $ in}& ~ L^\infty ((0,T);H^4 (\mathbb{T}^2)),  \\
    u_N & \rightarrow u & \text{strongly in}& ~ C([0,T];H^2(\mathbb T^2)), \\
    \sigma_N &\overset{\ast}{\rightharpoonup} \sigma & \text{weakly-$* $ in}& ~ L^\infty ((0,T);H^1 (\mathbb{T}^2)), \\
    \sigma_N &\rightarrow \sigma & \text{strongly in}& ~ C([0,T];L^2(\mathbb T^2)),\\
    \partial_t \sigma_N & \overset{\ast}{\rightharpoonup} \partial_t \sigma & \text{weakly-$* $ in}& ~ L^\infty ((0,T);H^1 (\mathbb{T}^2)).
\end{align}
It is straightforward to check that $(u,\sigma)$ is a solution of the regularised advective Kelvin-Voigt EVP system \eqref{regularisedkelvinvoigtsystem}. 
\subsection{Existence of a local strong solution for the advective Kelvin-Voigt EVP model \eqref{advectionsystem}} \label{limitsection}
In order to prove the existence of a solution to the advective Kelvin-Voigt EVP model \eqref{advectionsystem}, we will now obtain energy estimates on the regularised system \eqref{regularisedkelvinvoigtsystem} which are uniform in $\beta$ and $\delta$. By the existence result from Section \ref{regularisedsystemsection}, there exists a sequence $\{ (u_{\beta,\delta}, \sigma_{\beta,\delta} ) \} $ which solves the regularised system \eqref{regularisedkelvinvoigtsystem} for any $\beta, \delta > 0$ on a time interval of existence $[0,T_{\beta,\delta}]$. The existence time $T_{\beta,\delta}$ from Section \ref{regularisedsystemsection} is independent of $N$ and $\epsilon$, but it does depend on $\alpha$, $\beta$, $\delta$, $\mathcal{E}$, $P$, $u_0$, $\sigma_0$ and the parameters of the external forces. We need to show that these solutions have a uniform time of existence $T$ (i.e., independent of $\beta$ and $\delta$) and regularity bounds as $\beta,\delta \rightarrow 0$. By applying Lemma \ref{Lpestimatelemma}, we have the following uniform estimate
\begin{equation} \label{Linftyestimate}
\lVert \sigma_{\beta,\delta} (\cdot, t) \rVert_{L^\infty} \leq \lVert \sigma_0^\delta \rVert_{L^\infty} + 2 P.
\end{equation}
By proceeding analogously as in Section \ref{regularisedsystemsection}, one finds (for a time $t \in [0, T_{\beta,\delta}]$)
\begin{align*}
&\frac{1}{2} \frac{\textrm{d}}{\dt} \bigg[ \lVert u_{\beta,\delta} \rVert_{H^1}^2 + \mathcal{E}^{-1} \lVert \sigma_{\beta,\delta} \rVert_{H^1}^2 + \alpha^2 \lVert \nabla u_{\beta,\delta} \rVert_{H^1}^2 + \beta^4 \lVert \Delta u_{\beta,\delta} \rVert_{H^1}^2 \bigg] \\
&= \int_{\mathbb{T}^2} \bigg[ (\nabla \cdot \sigma_{\beta,\delta}) \cdot (u_{\beta,\delta} - \Delta u_{\beta,\delta}) + D(u_{\beta,\delta}) : (\sigma_{\beta,\delta} - \Delta \sigma_{\beta,\delta}) \bigg] \dx \\
&- \int_{\mathbb{T}^2} \big[(u_{\beta,\delta} \cdot \nabla) u_{\beta,\delta} \big] \cdot \big[u_{\beta,\delta} - \Delta u_{\beta,\delta} \big] \dx \\
&+ \int_{\mathbb{T}^2} \bigg[ \mathcal{T}_a + \mathcal{T}_w + \Omega u^\perp_{\beta,\delta} - g \nabla H_0 \big] \cdot \big[u_{\beta,\delta} - \Delta u_{\beta,\delta} \big] \dx \\
&- \int_{\mathbb{T}^2} \bigg[ \frac{4 \mathcal{D}_\epsilon}{P} (\sigma_{\beta,\delta} - \frac{1}{2} \Tr \sigma_{\beta,\delta} \mathbb I_2) + \frac{\mathcal{D}_\epsilon}{2 P} \Tr \sigma_{\beta,\delta} \mathbb I_2 + \frac{\mathcal{D}_\epsilon}{2} \mathbb I_2 \bigg] \cdot \big(\sigma_{\beta,\delta} - \Delta \sigma_{\beta,\delta} \big) \dx \\
&\eqqcolon I_5 + I_6 + I_7 + I_8.
\end{align*}
By using the same estimates as in Section \ref{regularisedsystemsection}, we obtain that
\begin{align*}
I_5 &= 0, \\
\lvert I_6 \rvert &\lesssim \lVert u_{\beta,\delta} \rVert_{H^2}^3, \\
\lvert I_7 \rvert &\lesssim \lVert u_{\beta,\delta} \rVert_{H^2} \bigg[ \lVert U_a \rVert_{L^4}^2 + \lVert U_w \rVert_{L^4}^2 + \lVert u_{\beta,\delta} \rVert_{L^4}^2 + \lVert H_0 \rVert_{H^1} \bigg].
\end{align*}
Similarly, we get
\begin{align*}
&\bigg\lvert \int_{\mathbb{T}^2} \bigg[ \frac{4 \mathcal{D}_\epsilon}{P} (\sigma_{\beta,\delta} - \frac{1}{2} \Tr \sigma_{\beta,\delta} \mathbb I_2) + \frac{\mathcal{D}_\epsilon}{2 P} \Tr \sigma_{\beta,\delta} \mathbb I_2 + \frac{\mathcal{D}_\epsilon}{2} \mathbb I_2 \bigg] \cdot \sigma_{\beta,\delta} \dx \bigg\rvert \\
&\lesssim (\lVert \nabla u_{\beta,\delta} \rVert_{L^2} + 1) (1 + \lVert \sigma_{\beta,\delta} \rVert_{L^4}^2).
\end{align*}
What remains to be shown is the estimate on the second term of $I_8$ (involving $\Delta \sigma_{\beta,\delta}$). We have
\begin{align*}
&\int_{\mathbb{T}^2} \bigg[ \frac{4 \mathcal{D}_\epsilon}{P} (\sigma_{\beta,\delta} - \frac{1}{2} \Tr \sigma_{\beta,\delta} \mathbb I_2) + \frac{\mathcal{D}_\epsilon}{2 P} \Tr \sigma_{\beta,\delta} \mathbb I_2 + \frac{\mathcal{D}_\epsilon}{2} \mathbb I_2 \bigg] \cdot \Delta \sigma_{\beta,\delta} \dx \\
&= - \int_{\mathbb{T}^2} \bigg[ \frac{4 \nabla \mathcal{D}_\epsilon}{P} (\sigma_{\beta,\delta} - \frac{1}{2} \Tr \sigma_{\beta,\delta} \mathbb I_2) + \frac{\nabla \mathcal{D}_\epsilon}{2 P} \Tr \sigma_{\beta,\delta} \mathbb I_2  + \frac{\nabla \mathcal{D}_\epsilon}{2} \mathbb I_2 \bigg] \cdot \nabla \sigma_{\beta,\delta} \dx \\
&- \int_{\mathbb{T}^2} \bigg[ \frac{4 \mathcal{D}_\epsilon}{P} \lvert \nabla (\sigma_{\beta,\delta} - \frac{1}{2} \Tr \sigma_{\beta,\delta} \mathbb I_2 ) \rvert^2 + \frac{\mathcal{D}_\epsilon}{2 P} \lvert \nabla (\Tr \sigma_{\beta,\delta} ) \rvert^2 \bigg] \dx.
\end{align*}
We note that the term in the third line is negative and hence can be disregarded in the energy estimate. By using inequality \eqref{Linftyestimate}, we can obtain the following bound on the term from the second line
\begin{align*}
&\bigg\lvert \int_{\mathbb{T}^2} \bigg[ \frac{4 \nabla \mathcal{D}_\epsilon}{P} (\sigma_{\beta,\delta} - \frac{1}{2} \Tr \sigma_{\beta,\delta} \mathbb I_2) + \frac{\nabla \mathcal{D}_\epsilon}{2 P} \Tr \sigma_{\beta,\delta} \mathbb I_2  + \frac{\nabla \mathcal{D}_\epsilon}{2} \mathbb I_2 \bigg] \cdot \nabla \sigma_{\beta,\delta} \dx \bigg\rvert \\
&\lesssim \lVert u_{\beta,\delta} \rVert_{H^2} (\lVert \sigma_{\beta,\delta} \rVert_{L^\infty} + 1) \lVert \sigma_{\beta,\delta} \rVert_{H^1} \lesssim \lVert u_{\beta,\delta} \rVert_{H^2} (\lVert \sigma_0^\delta \rVert_{L^\infty} + 1 + 2 P) \lVert \sigma_{\beta,\delta} \rVert_{H^1} \\
&\lesssim \lVert u_{\beta,\delta} \rVert_{H^2} (\lVert \sigma_0 \rVert_{L^\infty} + 1 + 2 P) \lVert \sigma_{\beta,\delta} \rVert_{H^1},
\end{align*}
where we have used Theorem C.16 in \cite{leoni} in order to bound $\lVert \sigma_0^\delta \rVert_{L^\infty}$ by $\lVert \sigma_0 \rVert_{L^\infty}$. Therefore, by combining the estimates on $I_5, I_6, I_7$ and $I_8$ we obtain that
\begin{align*}
&\frac{1}{2} \frac{\textrm{d}}{\dt} \bigg[ \lVert u_{\beta,\delta} \rVert_{H^1}^2 + \mathcal{E}^{-1} \lVert \sigma_{\beta,\delta} \rVert_{H^1}^2 + \alpha^2 \lVert \nabla u_{\beta,\delta} \rVert_{H^1}^2 + \beta^4 \lVert \Delta u_{\beta,\delta} \rVert_{H^1}^2 \bigg] \\
&\lesssim \lVert u_{\beta,\delta} \rVert_{H^2} \bigg[ \lVert U_a \rVert_{L^4}^2 + \lVert U_w \rVert_{L^4}^2 + \lVert u_{\beta,\delta} \rVert_{H^2}^2 + \lVert H_0 \rVert_{H^1} \bigg] + 1 + \lVert \sigma_{\beta,\delta} \rVert_{H^1}^2,
\end{align*}
where we note that the constants on the right-hand side of this estimate are independent of $\beta$ and $\delta$. Therefore, we find that there exists a time $T > 0$ and a constant $K \in (0,\infty)$, which do not depend on $\beta$, $\delta$, $\epsilon$ and $N$, such that
\begin{align*}
&\lVert u_{\beta,\delta} \rVert_{L^\infty ((0,T); H^2 (\mathbb{T}^2))} + \lVert \sigma_{\beta,\delta} \rVert_{L^\infty ((0,T); H^1 (\mathbb{T}^2))} + \lVert \sigma_{\beta,\delta} \rVert_{L^\infty ((0,T); L^\infty (\mathbb{T}^2))} \\
&+ \lVert \partial_t u_{\beta,\delta} \rVert_{L^\infty ((0,T); H^2 (\mathbb{T}^2))} + \lVert \partial_t \sigma_{\beta,\delta} \rVert_{L^\infty ((0,T); L^2 (\mathbb{T}^2))} \leq K < \infty.
\end{align*}
The time $T$ and the constant $K$ do depend on $\alpha$, $u_0$, $\sigma_0$, $c_a$, $c_w$, $\rho_a$, $\rho_w$, $\mathcal{E}$, $g$, $H_0$, $P$, $\phi$, $\theta$, $U_a$ and $U_w$.

Hence, by again applying the Aubin-Lions and Banach-Alaoglu compactness theorems, we find that as $\beta, \delta \rightarrow 0$ (by passing to a subsequence, if required)
\begin{align}
    u_{\beta,\delta} & \overset{\ast}{\rightharpoonup} u &  \text{weakly-$* $ in}& ~ L^\infty ((0,T);H^2 (\mathbb{T}^2)), \\
    \partial_t u_{\beta,\delta} & \overset{\ast}{\rightharpoonup} \partial_t u &  \text{weakly-$* $ in}& ~ L^\infty ((0,T);H^2 (\mathbb{T}^2)),  \\
    u_{\beta,\delta} & \rightarrow u & \text{strongly in}& ~ C([0,T];H^1(\mathbb T^2)), \\
    \sigma_{\beta,\delta} &\overset{\ast}{\rightharpoonup} \sigma & \text{weakly-$* $ in}& ~ L^\infty ((0,T);H^1 (\mathbb{T}^2)), \\
    \sigma_{\beta,\delta} &\rightarrow \sigma & \text{strongly in}& ~ C([0,T];L^p(\mathbb T^2)),\\
    \partial_t \sigma_{\beta,\delta} & \overset{\ast}{\rightharpoonup} \partial_t \sigma & \text{weakly-$* $ in}& ~ L^\infty ((0,T); L^2 (\mathbb{T}^2)),
\end{align}
for any $p \in [1, \infty)$. As all the estimates are uniform in $\epsilon$, one can take another limit and send $\epsilon \rightarrow 0$. One can check that the resulting limit $(u,\sigma)$ satisfies the advective Kelvin-Voigt EVP model \eqref{advectionsystem}, in particular it attains the initial conditions.
\subsection{Continuous dependence on the initial data and uniqueness of the strong solution} \label{uniquenesssection}
In order to conclude the proof of Theorem \ref{localwellposednessthm}, we need to show the uniqueness of the strong solution which we constructed in the previous sections. Suppose there exist two solutions $(u_1, \sigma_1)$ and $(u_2, \sigma_2)$ to the advective Kelvin-Voigt EVP model \eqref{advectionsystem}, which obey the initial conditions $(u_{1,0}, \sigma_{1,0})$ and $(u_{2,0}, \sigma_{2,0})$. Moreover, they satisfy $u_i \in C([0,T_i];H^2 (\mathbb{T}^2))$ and $\sigma_i \in C([0,T_i]; H^1_w (\mathbb{T}^2)) \cap C([0,T_i]; L^2 (\mathbb{T}^2))$ on the time intervals of existence $[0,T_i]$ for $i=1,2$. Then we consider the difference between the two solutions
\begin{equation*}
\delta u \coloneqq u_1 - u_2, \quad \delta \sigma \coloneqq \sigma_1 - \sigma_2. 
\end{equation*}
The difference $(\delta u, \delta \sigma)$ satisfies the following system of equations (on the time interval $[0, \min \{ T_1, T_2 \}]$)
\begin{align}
\partial_t (\delta u - \alpha^2 \Delta \delta u) + u_1 \cdot \nabla \delta u + \delta u \cdot \nabla u_2  = \nabla \cdot \delta \sigma + \mathcal{T}_{w,1} - \mathcal{T}_{w,2} + \Omega (\delta u)^\perp ,  \\
\frac{1}{\mathcal{E}} \partial_t \delta \sigma + \frac{4 \mathcal D_1}{P}(\sigma_1 - \frac{1}{2} \Tr \sigma_1 \mathbb I_2) - \frac{4 \mathcal D_2}{P}(\sigma_2 - \frac{1}{2} \Tr \sigma_2 &\mathbb I_2 ) \nonumber \\
+ \frac{\mathcal D_1}{2P} \Tr \sigma_1 \mathbb I_2 - \frac{\mathcal D_2}{2P} \Tr \sigma_2 \mathbb I_2 
+ \frac{\mathcal{D}_1 - \mathcal{D}_2}{2} &\mathbb I_2 = D(\delta u),
\end{align}
where in the above we have used the following notation (for $i=1,2$)
\begin{align}
\mathcal{T}_{w,i} &= c_w \rho_w \lvert U_w - u_i \rvert \big[ (U_w - u_i) \cos \theta + (U_w - u_i)^\perp \sin \theta \big], \\
\mathcal{D}_i &= \lvert D(u_i) \rvert.
\end{align} 
Then one can find the following equality (again for $t \in [0, \min \{ T_1, T_2 \}]$)
\begin{align*}
&\frac{1}{2} \frac{\rm d}{\dt} \bigg[ \lVert \delta u \rVert_{L^2}^2 + \alpha^2 \lVert \nabla \delta u \rVert_{L^2}^2 + \mathcal{E}^{-1} \lVert \delta \sigma \rVert_{L^2}^2 \bigg] \\
&= - \int_{\mathbb{T}^2} \big[ u_1 \cdot \nabla \delta u + \delta u \cdot \nabla u_2 \big] \cdot \delta u \dx - \int_{\mathbb{T}^2} \frac{\mathcal{D}_1 - \mathcal{D}_2}{2} \Tr \delta \sigma \dx \\
&- \int_{\mathbb{T}^2} \frac{4 (\mathcal{D}_1 - \mathcal{D}_2)}{P}(\sigma_1 - \frac{1}{2} \Tr \sigma_1 \mathbb I_2) : \delta \sigma \dx - \int_{\mathbb{T}^2} \frac{4 \mathcal D_2}{P}(\delta \sigma - \frac{1}{2} \Tr \delta \sigma \mathbb I_2) : \delta \sigma \dx \\
&- \int_{\mathbb T^2} \frac{(\mathcal{D}_1 - \mathcal{D}_2)}{2P} \Tr \sigma_1 \mathbb I_2 : \delta \sigma \dx - \int_{\mathbb T^2} \frac{\mathcal D_2}{2P} \Tr \delta \sigma \mathbb I_2 : \delta \sigma \dx + \int_{\mathbb{T}^2} \big( \mathcal{T}_{w,1} - \mathcal{T}_{w,2} \big) \cdot \delta u \dx \\
&\eqqcolon J_1 + J_2 + J_3 + J_4 + J_5 + J_6 + J_7.
\end{align*}
We now estimate the contributions from the various terms $J_1, \ldots, J_7$. By using Lemma \ref{strainratelemma}, we have
\begin{align*}
\lvert J_1 \rvert &\leq \lVert u_1 \rVert_{L^\infty} \lVert \nabla \delta u \rVert_{L^2} \lVert \delta u \rVert_{L^2} + \lVert \delta u \rVert_{L^4} \lVert \nabla u_2 \rVert_{L^4} \lVert \delta u \rVert_{L^2}, \\
\lvert J_2 \rvert &\leq \lVert \mathcal{D}_1 - \mathcal{D}_2 \rVert_{L^2} \lVert \delta \sigma \rVert_{L^2} \leq \lVert \nabla \delta u \rVert_{L^2} \lVert \delta \sigma \rVert_{L^2}, \\
\lvert J_3 \rvert &\leq \bigg\lvert \int_{\mathbb{T}^2} \bigg[ \frac{4 (\mathcal{D}_1 - \mathcal{D}_2)}{P}(\sigma_1 - \frac{1}{2} \Tr \sigma_1 \mathbb I_2)  \bigg] : \delta \sigma \dx \bigg\rvert \lesssim \lVert \nabla \delta u \rVert_{L^2} \lVert \sigma_1 \rVert_{L^\infty} \lVert \delta \sigma \rVert_{L^2}, \\
J_4 &= - \int_{\mathbb{T}^2} \frac{4 \mathcal D_2}{P} \left\lvert \delta \sigma - \frac{1}{2} \Tr \delta \sigma \mathbb I_2 \right\rvert^2 \dx \leq 0, \\
\lvert J_5 \rvert &\leq \lVert \mathcal{D}_1 - \mathcal{D}_2 \rVert_{L^2} \lVert \sigma_1 \rVert_{L^\infty} \lVert \delta \sigma \rVert_{L^2} \lesssim \lVert \sigma_1 \rVert_{L^\infty} \lVert \nabla \delta u \rVert_{L^2} \lVert \delta \sigma \rVert_{L^2}, \\
J_6 &= - \int_{\mathbb T^2} \frac{\mathcal D_2}{2P} \lvert \Tr \delta \sigma \rvert^2 \dx, \\
\lvert J_7 \rvert &= \bigg\lvert - \int_{\mathbb{T}^2} c_w \rho_w \lvert U_w - u_1 \rvert \bigg[ \delta u \cos \theta + (\delta u)^\perp \sin \theta \bigg] \cdot \delta u \dx \\
&+ \int_{\mathbb{T}^2} c_w \rho_w \big[ \lvert U_w - u_1 \rvert - \lvert U_w - u_2 \rvert \big] \bigg[ (U_w - u_2) \cos \theta + (U_w - u_2)^\perp \sin \theta \bigg] \cdot \delta u \dx \bigg\rvert \\
&\lesssim (\lVert U_w \rVert_{L^\infty} + \lVert u_1 \rVert_{L^\infty} + \lVert u_2 \rVert_{L^\infty} ) \lVert \delta u \rVert_{L^2}^2.
\end{align*}
We therefore obtain the following estimate
\begin{align}
&\frac{1}{2} \frac{\rm d}{\dt} \bigg[ \lVert \delta u \rVert_{L^2}^2 + \alpha^2 \lVert \nabla \delta u \rVert_{L^2}^2 + \mathcal{E}^{-1} \lVert \delta \sigma \rVert_{L^2}^2 \bigg] \label{continuousdependence} \\
&= \bigg[ \lVert U_w \rVert_{L^\infty} + \lVert u_1 \rVert_{H^2} + \lVert u_2 \rVert_{H^2} + \lVert \sigma_1 \rVert_{L^\infty} + \lVert \sigma_2 \rVert_{L^\infty} \bigg] \cdot \bigg[ \lVert \delta u \rVert_{L^2}^2 + \alpha^2 \lVert \nabla \delta u \rVert_{L^2}^2 + \mathcal{E}^{-1} \lVert \delta \sigma \rVert_{L^2}^2 \bigg], \nonumber
\end{align}
from which the uniqueness of strong solutions to the advective Kelvin-Voigt EVP model \eqref{advectionsystem} follows by Grönwall's inequality. Note that estimate \eqref{continuousdependence} also shows the continuous dependence of the solution on the initial data. 

\section{Proof of Theorem \ref{globalwellposednessthm}: The Kelvin-Voigt EVP model \eqref{kelvinvoigtsystem} (without advection)}
As was done in the proof of Theorem \ref{localwellposednessthm}, for the proof of Theorem \ref{globalwellposednessthm} we will first regularise the strain rate and introduce an additional term $\beta^4 \partial_t \Delta^2 u$ in the momentum balance. Therefore we consider the following approximate system (and we mollify the initial data in a similar manner to \eqref{mollification})
\begin{align}
&\partial_t (u_{\beta,\delta} - \alpha^2 \Delta u_{\beta,\delta} + \beta^4 \Delta^2 u_{\beta,\delta}) = \nabla \cdot \sigma_{\beta,\delta} + \mathcal{T}_a + \mathcal{T}_w + \Omega u_{\beta,\delta}^\perp - g \nabla H_0, \\
&\dfrac{1}{\mathcal E} \partial_t \sigma_{\beta,\delta} + \dfrac{4 \mathcal{D}_\epsilon}{P} (\sigma_{\beta,\delta} - \frac{1}{2} \Tr \sigma_{\beta,\delta} \mathbb I_2) + \dfrac{\mathcal{D}_\epsilon}{2P} \Tr \sigma_{\beta,\delta} \mathbb I_2 + \dfrac{\mathcal{D}_\epsilon}{2} \mathbb I_2  = D(u_{\beta,\delta}),
	\\
&u \lvert_{t = 0} = u_0^\delta, \quad \sigma \lvert_{t = 0} = \sigma_0^\delta.
\end{align}
By proceeding in a completely analogous manner as in the proof of Theorem \ref{localwellposednessthm}, we deduce that there exists a unique local-in-time solution $(u,\sigma) \in C([0,T];H^2 (\mathbb{T}^2)) \times C([0,T];H^1 (\mathbb{T}^2))$ (in fact $u \in C([0,T]; H^3 (\mathbb{T}^2))$), where $T$ is independent of $N$ and $\epsilon$, but it depends on $\alpha$, $\beta$, $\delta$, $u_0$, $\sigma_0$, $c_a$, $c_w$, $\rho_a$, $\rho_w$, $\mathcal{E}$, $g$, $H_0$, $P$, $\phi$, $\theta$, $U_a$ and $U_w$. We need to show that the solution $(u_{\beta,\delta},\sigma_{\beta,\delta})$ is global-in-time and satisfies estimates which are uniform in $\beta$ and $\delta$. More precisely, in what follows we show the following uniform estimates
\begin{align}
&\frac{1}{2} \frac{\textrm{d}}{\dt} \bigg[ \lVert u_{\beta,\delta} \rVert_{L^2}^2 + \mathcal{E}^{-1} \lVert \tau_{\beta,\delta} \rVert_{L^2}^2 + \alpha^2 \lVert \nabla u_{\beta,\delta} \rVert_{L^2}^2 + \beta^4 \lVert \Delta u_{\beta,\delta} \rVert_{L^2}^2 \bigg] \lesssim 1 + \lVert u_{\beta,\delta} \rVert_{L^2}^2, \label{L2energyestimate} \\
&\lVert \sigma_{\beta,\delta} (\cdot, t) \rVert_{L^\infty} \leq \lVert \sigma_0 \rVert_{L^\infty} + 2 P, \label{Linftyenergyestimate} \\
&\frac{1}{2} \frac{\textrm{d}}{\dt} \bigg[ \lVert \nabla u_{\beta,\delta} \rVert_{L^2}^2 + \mathcal{E}^{-1} \lVert \nabla \sigma_{\beta,\delta} \rVert_{L^2}^2 + \alpha^2 \lVert \Delta u_{\beta,\delta} \rVert_{L^2}^2 + \beta^4 \lVert \Delta \nabla u_{\beta,\delta} \rVert_{L^2}^2 \bigg] \nonumber \\
&\lesssim \big(\lVert \sigma_{\beta,\delta} \rVert_{L^\infty} + \lVert u_{\beta,\delta} \rVert_{L^2} + 1\big) \big(1 + \lVert u_{\beta,\delta} \rVert_{H^2}^2 + \lVert \sigma_{\beta,\delta} \rVert_{H^1}^2 \big). \label{H1energyestimate}
\end{align}
Note that the $L^\infty (\mathbb{T}^2)$-estimate \eqref{Linftyenergyestimate} will follow from Lemma \ref{Lpestimatelemma}. Observe that these estimate will first be established on the time interval $[0,T]$, but then the global existence of the solution will follow from these bounds. We will first prove the $L^2 (\mathbb{T}^2)$-estimate \eqref{L2energyestimate}. By using the unknown $\tau$ defined in \eqref{newunkown} (which satisfies equation \eqref{modifiedstressequation}), we obtain the following $L^2 (\mathbb{T}^2)$-estimate
\begin{align*}
&\frac{1}{2} \frac{\textrm{d}}{\dt} \bigg[ \lVert u_{\beta,\delta} \rVert_{L^2}^2 + \mathcal{E}^{-1} \lVert \tau_{\beta,\delta} \rVert_{L^2}^2 + \alpha^2 \lVert \nabla u_{\beta,\delta} \rVert_{L^2}^2 + \beta^4 \lVert \Delta u_{\beta,\delta} \rVert_{L^2}^2 \bigg] \\
&= \int_{\mathbb{T}^2} \bigg[ (\nabla \cdot \tau_{\beta,\delta}) \cdot u_{\beta,\delta} + D(u_{\beta,\delta}) : \tau_{\beta,\delta} \bigg] \dx \\
&+ \int_{\mathbb{T}^2} \bigg[ \mathcal{T}_a + \mathcal{T}_w + \Omega u^\perp_{\beta,\delta} - g \nabla H_0 \big] \cdot u_{\beta,\delta} \dx \\
&- \int_{\mathbb{T}^2} \bigg[ \frac{4 \mathcal{D}_\epsilon}{P} (\tau_{\beta,\delta} - \frac{1}{2} \Tr \tau_{\beta,\delta} \mathbb I_2) + \frac{\mathcal{D}_\epsilon}{2 P} \Tr \tau_{\beta,\delta} \mathbb I_2 \bigg] : \tau_{\beta,\delta} \dx \\
&\eqqcolon I_9 + I_{10} + I_{11}.
\end{align*}
We then have 
\begin{align*}
I_9 &= 0, \\
I_{10} &\lesssim \int_{\mathbb{T}^2} \mathcal{T}_w \cdot u_{\beta,\delta} \dx + 1 + \lVert u_{\beta,\delta} \rVert_{L^2}^2, \\
I_{11} &= - \int_{\mathbb{T}^2} \bigg[ \frac{4 \mathcal{D}_\epsilon}{P} \bigg\lvert \tau_{\beta,\delta} - \frac{1}{2} \Tr \tau_{\beta,\delta} \mathbb I_2 \bigg\rvert^2 + \frac{\mathcal{D}_\epsilon}{2 P} \lvert \Tr \tau_{\beta,\delta} \rvert^2 \bigg] \dx \leq 0.
\end{align*}
In the case of $I_9$ we have used the divergence theorem. In order to bound the right-hand side of the estimate for $I_{10}$ we estimate the contribution from the oceanic drag force
\begin{align*}
\int_{\mathbb{T}^2} \mathcal{T}_w \cdot u_{\beta,\delta} \dx &= - c_w \rho_w \int_{\mathbb{T}^2} \lvert U_w - u_{\beta,\delta} \rvert^3 \cos \theta \dx \\
&+ c_w \rho_w \int_{\mathbb{T}^2} \lvert U_w - u_{\beta,\delta} \rvert \bigg[ (U_w - u_{\beta,\delta}) \cdot U_w \cos \theta + U_w^\perp \cdot u_{\beta,\delta} \sin \theta \bigg] \dx \\
&\lesssim 1 + \lVert u_{\beta,\delta} \rVert_{L^2}^2.
\end{align*}
Therefore we can conclude that
\begin{align} \label{L2estimate}
\frac{1}{2} \frac{\textrm{d}}{\dt} \bigg[ \lVert u_{\beta,\delta} \rVert_{L^2}^2 + \mathcal{E}^{-1} \lVert \tau_{\beta,\delta} \rVert_{L^2}^2 + \alpha^2 \lVert \nabla u_{\beta,\delta} \rVert_{L^2}^2 + \beta^4 \lVert \Delta u_{\beta,\delta} \rVert_{L^2}^2 \bigg] \lesssim 1 + \lVert u_{\beta,\delta} \rVert_{L^2}^2.
\end{align}
From this it follows that we have a uniform bound on $\lVert u_{\beta,\delta} \rVert_{H^1}$ and $\lVert \sigma_{\beta,\delta} \rVert_{L^2}$ with respect to $\beta$ and $\delta$. Then by applying Lemma \ref{Lpestimatelemma} we obtain a uniform estimate on $\lVert \sigma_{\beta,\delta} \rVert_{L^\infty}$. Next we turn to the proof of the $\dot{H}^1 (\mathbb{T}^2)$-estimate \eqref{H1energyestimate}. We have
\begin{align*}
&\frac{1}{2} \frac{\textrm{d}}{\dt} \bigg[ \lVert \nabla u_{\beta,\delta} \rVert_{L^2}^2 + \mathcal{E}^{-1} \lVert \nabla \sigma_{\beta,\delta} \rVert_{L^2}^2 + \alpha^2 \lVert \Delta u_{\beta,\delta} \rVert_{L^2}^2 + \beta^4 \lVert \Delta \nabla u_{\beta,\delta} \rVert_{L^2}^2 \bigg] \\
&+ \int_{\mathbb{T}^2} \bigg[ \frac{4 \mathcal{D}_\epsilon}{P} \lvert \nabla (\sigma_{\beta,\delta} - \frac{1}{2} \Tr \sigma_{\beta,\delta} \mathbb I_2 ) \rvert^2 + \frac{\mathcal{D}_\epsilon}{2 P} \lvert \nabla \Tr \sigma_{\beta,\delta} \rvert^2 \bigg] \dx \\
&= \sum_{i=1}^2 \int_{\mathbb{T}^2} \bigg[ (\nabla \cdot \partial_i \sigma_{\beta,\delta}) \cdot \partial_i u_{\beta,\delta} + D(\partial_i u_{\beta,\delta}) : \partial_i \sigma_{\beta,\delta} \bigg] \dx \\
&+ \int_{\mathbb{T}^2} \bigg[ \mathcal{T}_a + \mathcal{T}_w + \Omega u^\perp_{\beta,\delta} - g \nabla H_0 \big] \cdot - \Delta u_{\beta,\delta} \dx \\
&- \sum_{i=1}^2 \int_{\mathbb{T}^2} \bigg[ \frac{4 \partial_i \mathcal{D}_\epsilon}{P} (\sigma_{\beta,\delta} - \frac{1}{2} \Tr \sigma_{\beta,\delta} \mathbb I_2) + \frac{\partial_i \mathcal{D}_\epsilon}{2 P} \Tr \sigma_{\beta,\delta} \mathbb I_2 + \frac{\partial_i \mathcal{D}_\epsilon}{2} \mathbb I_2 \bigg] \cdot \partial_i \sigma_{\beta,\delta} \dx \\
&\eqqcolon I_{12} + I_{13} + I_{14}.
\end{align*}
By the divergence theorem we have $I_{12} = 0$. Moreover, we have
\begin{align*}
\lvert I_{13} \rvert &\lesssim \big(\lVert U_a \rVert_{L^4}^2 + \lVert U_w \rVert_{L^4}^2 + \lVert u_{\beta,\delta} \rVert_{L^4}^2 + \lVert u_{\beta,\delta} \rVert_{L^2} + \lVert H_0 \rVert_{H^1} \big) \lVert \Delta u_{\beta,\delta} \rVert_{L^2}, \\
\lvert I_{14} \rvert &\lesssim \lVert \Delta u_{\beta,\delta} \rVert_{L^2} \big(\lVert \sigma_{\beta,\delta} \rVert_{L^\infty} + 1\big) \lVert \nabla \sigma_{\beta,\delta} \rVert_{L^2}.
\end{align*}
We therefore conclude that
\begin{align} \nonumber
&\frac{1}{2} \frac{\textrm{d}}{\dt} \bigg[ \lVert \nabla u_{\beta,\delta} \rVert_{L^2}^2 + \mathcal{E}^{-1} \lVert \nabla \sigma_{\beta,\delta} \rVert_{L^2}^2 + \alpha^2 \lVert \Delta u_{\beta,\delta} \rVert_{L^2}^2 + \beta^4 \lVert \Delta \nabla u_{\beta,\delta} \rVert_{L^2}^2 \bigg] \\
&\lesssim \big(\lVert \sigma_{\beta,\delta} \rVert_{L^\infty} + \lVert u_{\beta,\delta} \rVert_{L^2} + 1\big) \big(1 + \lVert u_{\beta,\delta} \rVert_{H^2}^2 + \lVert \sigma_{\beta,\delta} \rVert_{H^1}^2 \big). \label{H1estimate}
\end{align}
Then by applying the Grönwall inequality to estimates \eqref{L2estimate} and \eqref{H1estimate} (as well as using Lemma \ref{Lpestimatelemma}) we have
\begin{align} \label{uniformestimate}
\sup_{t \in [0,T]} \biggl\lbrack \lVert u_{\beta,\delta} \rVert_{H^2} + \lVert \sigma_{\beta,\delta} \rVert_{H^1} + \lVert \sigma_{\beta,\delta} \rVert_{L^\infty} + \Vert \partial_t u_{\beta,\delta} \Vert_{H^2} + \Vert \partial_t\sigma_{\beta,\delta} \Vert_{L^2} \biggr\rbrack \leq C < \infty,
\end{align}
where we observe that the constant $C$ is independent of $\beta, \delta$ and $\epsilon$, while it depends on $\alpha$, $\mathcal{E}$ and $T$. Then by applying the Aubin-Lions lemma and the Banach-Alaoglu theorem we obtain the following convergence results as $\beta, \delta \rightarrow 0$ (by passing to subsequences, if required)
\begin{align}
    u_{\beta,\delta} & \overset{\ast}{\rightharpoonup} u &  \text{weakly-$* $ in}& ~ L^\infty ((0,T);H^2 (\mathbb{T}^2)), \\
    \partial_t u_{\beta,\delta} & \overset{\ast}{\rightharpoonup} \partial_t u &  \text{weakly-$* $ in}& ~ L^\infty ((0,T);H^2 (\mathbb{T}^2)),  \\
    u_{\beta,\delta} & \rightarrow u & \text{strongly in}& ~ C([0,T];H^1(\mathbb T^2)), \\
    \sigma_{\beta,\delta} &\overset{\ast}{\rightharpoonup} \sigma & \text{weakly-$* $ in}& ~ L^\infty ((0,T);H^1 (\mathbb{T}^2)), \\
    \sigma_{\beta,\delta} &\rightarrow \sigma & \text{strongly in}& ~ C([0,T];L^p(\mathbb T^2)),\\
    \partial_t \sigma_{\beta,\delta} & \overset{\ast}{\rightharpoonup} \partial_t \sigma & \text{weakly-$* $ in}& ~ L^\infty ((0,T); L^2 (\mathbb{T}^2)),
\end{align}
where $u \in C([0,T];H^2 (\mathbb{T}^2))$, $\sigma \in C([0,T];H^1_w (\mathbb{T}^2)) \cap C([0,T]; L^2 (\mathbb{T}^2))$ and $p \in [1,\infty)$. It easily follows that the limit $(u,\sigma)$ satisfies the regularised Kelvin-Voigt EVP model \eqref{kelvinvoigtsystemepsilon} with the regularised strain rate $\mathcal{D}_\epsilon$. Using similar estimates as in Section \ref{uniquenesssection}, we know that the solution $(u,\sigma)$ is the unique strong solution of the intermediate system \eqref{kelvinvoigtsystemepsilon}.

Then because we again have the uniform regularity estimates \eqref{L2energyestimate}-\eqref{H1energyestimate} (which are independent of $\epsilon$) for a sequence of solutions $(u_\epsilon ,\sigma_\epsilon )$ to \eqref{kelvinvoigtsystemepsilon} with the regularised strain rate $\mathcal{D}_\epsilon$ (as $\epsilon \rightarrow 0$)
\begin{equation}
\sup_{t \in [0,T]} \biggl\lbrack \lVert u_\epsilon \rVert_{H^2} + \lVert \sigma_\epsilon \rVert_{H^1} + \lVert \sigma_\epsilon \rVert_{L^\infty} + \Vert \partial_t u_\epsilon \Vert_{H^2} + \Vert \partial_t\sigma_\epsilon \Vert_{L^2} \biggr\rbrack \leq C < \infty.
\end{equation}
Therefore, we can pass to the limit $\epsilon \rightarrow 0$ in a similar fashion as before and find a solution $u \in C([0,T];H^2 (\mathbb{T}^2))$ and $\sigma \in C([0,T];H^1_w (\mathbb{T}^2)) \cap C([0,T]; L^2 (\mathbb{T}^2))$. Moreover, by proceeding in a similar manner as in Section \ref{uniquenesssection} we obtain that the constructed strong solution $(u,\sigma)$ is unique (and depends continuously on the initial data), which concludes the proof of the global well-posedness of the Kelvin-Voigt EVP model \eqref{kelvinvoigtsystem}.

\section{Conclusion}
In this paper, we have studied the EVP model with a Kelvin-Voigt regularisation of the momentum balance. We have considered two cases, namely both the presence and absence of an advection term in the momentum balance. In Theorem \ref{localwellposednessthm} we have proved the local well-posedness of the advective Kelvin-Voigt EVP model, while in Theorem \ref{globalwellposednessthm} we have proved the global well-posedness of the Kelvin-Voigt EVP model. A crucial new idea of the proof was the $L^\infty$-estimate which we established in Lemma \ref{Lpestimatelemma}, which in turn made it possible to obtain the $H^1$-estimate. 

From the modelling point of view, an advantage of the addition of the (Kelvin-)Voigt-regularisation to the momentum balance instead of the constitutive relation is that it is a more straightforward and clear modification of the rheology. The fact that we regularise in the unknown in which there is a loss of derivative allows us in Theorem \ref{globalwellposednessthm} to prove the existence of global strong solutions for much less regular initial data compared to \cite{boutros2025}. Moreover, we are able to handle the case of strain rates without cutoff (i.e., passing to the limit $\epsilon \rightarrow 0$) in the presence of an advection term in Theorem \ref{localwellposednessthm}, which was not possible with the estimates obtained in \cite{boutros2025} which only allowed the treatment of this limit in the absence of the advection term. However, as noted in \cite{boutros2025}, one can prove the local well-posedness of the Voigt-EVP model (i.e. the case where the constitutive relation includes the Voigt regularisation) with the advection term for any $\epsilon > 0$. The treatment of the full EVP system with a Voigt regularisation (i.e., including the hyperbolic balance laws for the mean ice thickness $h$ and ice compactness $A$) is left to future work. 

\section*{Acknowledgements}
The authors would like to thank Elizabeth Hunke for useful discussions at Texas A\&M University, and for her guidance through the various models used by practitioners. D.W.B. acknowledges support from the Cambridge Trust and the Cantab Capital Institute for Mathematics of Information. D.W.B. and E.S.T. have benefitted from the inspiring environment of the CRC 1114 ``Scaling Cascades in Complex Systems'', Project Number 235221301, Project C09, funded by the Deutsche Forschungsgemeinschaft (DFG). M.T. also acknowledges the funding by the DFG within the CRC 1114 ``Scaling Cascades in Complex Systems'', Project Number 235221301, Project B09. Moreover, this work was also supported in part by the DFG Research Unit FOR 5528 on Geophysical Flows. D.W.B. and M.T. would like to acknowledge the kind hospitality of the Department of Mathematics, Texas A\&M University, and M.T. also acknowledges the generous hospitality of the Department of Applied Mathematics and Theoretical Physics, University of Cambridge, where part of this work was completed.

\footnotesize
\bibliographystyle{acm}
\bibliography{main}

\end{document}